\documentclass[12pt]{amsart}
\usepackage{amscd,amsmath,bm,latexsym,amssymb,amscd}
\usepackage[mathscr]{euscript} 
\usepackage[all]{xy}
\usepackage{layout}
\usepackage{pb-diagram} 
\usepackage{pstricks}
\usepackage{enumerate}
\usepackage{dsfont}
\usepackage{verbatim}  

\title[On the moduli spaces of commuting elements in the projective unitary groups]{On the moduli spaces of commuting 
elements in the projective unitary groups}

\author[A.~Adem]{Alejandro Adem}
\address{Department of Mathematics,
University of British Columbia, Vancouver BC V6T 1Z2, Canada}
\email{adem@math.ubc.ca}

\author[M.C. Cheng]{Man Chuen Cheng}
\address{Department of Mathematics, The Chinese University of Hong Kong,
Shatin, N.T., Hong Kong}
\email{mccheng@math.cuhk.edu.hk}

\newcommand{\bZ}{\mathbb{Z}}

\newcommand{\bQ}{\mathbb{Q}}
\newcommand{\bC}{\mathbb{C}}
\newcommand{\bS}{\mathbb{S}}
\newcommand{\bR}{\mathbb{R}}
\newcommand{\bT}{\mathbb{T}}

\newcommand{\Tnm}{T(n,\bZ/m\bZ)}
\newcommand{\TnQZ}{T(n,\bQ/\bZ)}

\newcommand{\Hom}{\text{Hom}}
\newcommand{\Rep}{\text{Rep}}
\newcommand{\bigslant}[2]{{\left.\raisebox{.2em}{$#1$}\middle/\raisebox{-.2em}{$#2$}\right.}}

\newtheorem{theorem}{Theorem}[section]
\newtheorem{lemma}[theorem]{Lemma}
\newtheorem{corollary}[theorem]{Corollary}
\newtheorem{proposition}[theorem]{Proposition}
\newtheorem*{theorem*}{Theorem A}
\newtheorem*{theorem**}{Theorem B}
\newtheorem*{theorem***}{Theorem C}
\theoremstyle{definition}
\newtheorem{definition}[theorem]{Definition}
\newtheorem{example}[theorem]{Example}
\newtheorem{remark}[theorem]{Remark}

\DeclareMathOperator{\Tor}{Tor}

\def\quotient#1#2{%
    \raise1ex\hbox{$#1$}\Big/\lower1ex\hbox{$#2$}%
}

\begin{document}

\begin{abstract}
We provide descriptions for the moduli spaces $\Rep(\Gamma, PU(m))$, where $\Gamma$ is any finitely generated abelian group and  $PU(m)$ is the group of 
$m\times m$ projective unitary matrices. As an application we show that for any connected 
CW--complex $X$ with
$\pi_1(X)\cong \bZ^n$, the natural
map $\pi_0(\Rep(\pi_1(X), PU(m)))\to [X, BPU(m)]$ 
is injective, hence providing a complete enumeration of the isomorphism classes of flat principal
$PU(m)$--bundles over $X$. 
\end{abstract}

\maketitle
\section{Introduction}

The space of ordered commuting $n$--tuples in a compact, connected Lie group $G$ is by definition the subspace
$\Hom(\bZ^n, G)\subset G^n$ (see \cite{AC} for background and basic properties). 
Its orbit space under conjugation, denoted $Rep(\bZ^n, G)$, can be identified with the moduli space of isomorphism classes
of flat connections on principal $G$--bundles over the $n$--dimensional torus $\bT^n$.
In the case when all of the maximal abelian subgroups of $G$ 
are path connected, it can be shown that this moduli space has a single
connected component, corresponding to the identity element $(1, \dots , 1)\in G^n$ (see \cite{AC}, Proposition 2.3).
For example, $Rep(\bZ^n, U(m))\cong SP^m(\bT^n)$, the $m$-fold symmetric product of the $n$--torus.
However by a result due to Borel (see \cite{Borel}, page 216), for any prime number $p$ 
the fundamental group of $G$ has $p$--torsion
if and
only if there exists a rank two elementary abelian $p$--subgroup of $G$
which is not a subgroup of any torus. In this case $Rep(\bZ^n, G)$ fails to be path--connected
for all $n\ge 2$ and determining the number and exact structure of the 
components can be fairly complicated. Borel also shows that $H^*(G,\bZ)$ has $p$--torsion if and only if
there exists a subgroup of the form $(\bZ/p\bZ)^3\subset G$ which is not contained in any torus.
This can be used to show for example that $Rep(\bZ^3, Spin (7))$ is not path connected, even though
$Spin(7)$ is simply connected.


\medskip

In this note we consider the case when
$G=PU(m)$, the group of $m \times m$ projective unitary matrices, which has fundamental group
isomorphic to $\bZ/m\bZ$ and $$H^2(BPU(m), \bZ/m\bZ)\cong \bZ/m\bZ$$ has a canonical 
generator $\nu$ of order $m$, corresponding to the central extension
$$1\to \bZ/m\bZ \to SU(m)\to PU(m)\to 1.$$
Given a homomorphism
$h:\bZ^n\to PU(m)$, one can associate to it the cohomology class $\alpha=h^*(\nu)\in H^2(\bZ^n,\bZ/m\bZ)$. 
For our purposes it's convenient to identify 
$H^2(\bZ^n,\bZ/m\bZ)$ with the set $T(n,\bZ/m\bZ)$ of $n\times n$ skew--symmetric matrices over $\bZ/m\bZ$: given a basis $z_1,\dots , z_n$ for $\bZ^n$, then $D= (d_{ij})\in T(n,\bZ/m\bZ)$ corresponds to $\sum_{i<j} d_{ij}z_i^*z_j^*$. 

\medskip

Now given a skew--symmetric $n\times n$ matrix $D$ over $\bZ/m\bZ$ representing $\alpha$,
we define
$\sigma (\alpha) = \sigma (D) = \sqrt{|R(D)|}$, where 
$R(D)\subset (\bZ/m\bZ)^n$ 
is the row space of $D$ (see \cite{ACh}, Definition 2). Alternatively, for $\alpha\in H^2(\bZ^n, \bZ/m\bZ)$, by \cite{RV}, Proposition 4.1 we can find a basis
$g_1,\dots , g_n$ of $\bZ^n$ such that 
$\alpha = c_1 g_1^*g_2^* + \dots + c_r g_{2r-1}^*g_{2r}^*$ 
where $2r\le n$, $c_1,\dots, c_r\in \bZ/m\bZ$ and $|c_r| ~\Big{|}~ |c_{r-1}|~ \Big{|}~\dots ~\Big{|}~ |c_1|$. 
Using this basis it follows that  
$\sigma (\alpha) = \prod_{i=1}^r |c_i|$.

For a topological group $G$, 
let $\widehat{SP^n}(G)$ denote the \textbf{reduced} $n$-fold symmetric product of $G$, defined
as the quotient $SP^n(G)/G$, where $G$ acts by translation on each unordered coordinate. 
Then our main
result can be stated as follows:

\begin{theorem*}\label{thm:A}
For all $m,n \ge 1$ there are homeomorphisms 
$$Rep(\bZ^n, PU(m)) \cong \coprod\limits_{\substack{D\in T(n,\bZ/m\bZ)\\ \sigma(D)\big{|} m}}
\widehat{SP^{m\over\sigma(D)}}(\bT^n/R(D))$$
\end{theorem*}
Note that $\bT^n/R(D)\cong \bT^n$, so our formula expresses the moduli space as a disjoint union
of reduced symmetric products of the $n$--torus.   
The number of path-connected components of $Rep(\bZ^n, PU(m))$ is equal to
$N(n,m)=|\{D\in T(n,\bZ/m\bZ) ~ | ~ \sigma(D)\text{ divides }m\}|$,
a rather intricate number that has been computed in Adem-Cheng (see \cite{ACh}, Corollary 3.9). 
We also show how to apply our methods to provide a description of $Rep(\Gamma, PU(m))$ for any finitely generated abelian group $\Gamma$
(see Theorem \ref{thm:repgammaum}).
Note that by the results in \cite{Bergeron}, for these groups there is a homotopy equivalence
$Rep(\Gamma, PGL(m,\bC)) \simeq Rep (\Gamma, PU(m))$. 


\medskip

Recall that for $G$ a topological group and $X$ a CW--complex, the homotopy classes of maps
from $X$ to $BG$, denoted $[X,BG]$, classify isomorphism classes of principal $G$ bundles
over $X$. Representations play a key in this through the theory of flat bundles; in our setting the key connection
is via the induced map on components 
$$\Psi^{PU(m)}_{\bT^n}: \pi_0(Rep(\bZ^n, PU(m))) \to [\bT^n, BPU(m)].$$ 
Taking composition with 
the classifying map
$c_X: X\to B\pi_1(X)$, we obtain the following general result 

\begin{theorem**}
Let $X$ denote a connected CW--complex with $\pi_1(X) = \bZ^n$; then the 
map $$\Psi^{PU(m)}_{X}: \pi_0(Rep(\pi_1(X), PU(m))) \to [X, BPU(m)]$$ 
is injective for all $n,m \ge 1$ and so there are 
$N(n,m)$ distinct isomorphism classes of flat $PU(m)$--bundles on $X$.
\end{theorem**}

\noindent Regarding surjectivity, we obtain that for all 
$m\geq 2$, $$\Psi^{PU(m)}_{\bT^n}: \pi_0(Rep(\bZ^n, PU(m))) \to [\bT^n, BPU(m)]$$ 
is surjective if and
only if $n\leq 3$ (Proposition \ref{prop:surj}). It follows 
that there exists a principal $PU(m)$-bundle on the $n$--torus $\bT^n$ which does not admit 
a flat structure if and only if $n\geq 4$. Here we apply the results in \cite{W}, which
provide a classification of principal $PU(m)$--bundles over low--dimensional complexes.

\medskip

The results in this paper can be viewed as an application and refinement of the analysis carried out 
in our previous work \cite{ACh}, where we described the space 
$B_n(U(m))$ of almost commuting $n$--tuples in $U(m)$. 

\section{Projective Representations and Almost Commuting Elements}

In this section we will apply the methods from \cite{ACh} to give a description of $Rep(\bZ^n, PU(m))$. The almost
commuting elements will play a crucial role.

\begin{definition}
We define $B_n(G)$, the almost commuting elements in a Lie group $G$, as the set of all ordered $n$-tuples $(A_1,A_2,\ldots,A_n)\in G^n$ such that the commutators $[A_i,A_j]\in Z(G)$, the centre of $G$, for all $1\leq i,j\leq n$. 
\end{definition}

\medskip

Let $F_n$ denote the free group on $n$ generators $a_1,\ldots a_n$. We will identify a map from $F_n$ to $G$ with the $n$--tuple of images of these generators in $G$. Suppose that $f:F_n\to U(m)$ is in $B_n(U(m))$. For any $u,v\in F_n$, $[f(u),f(v)]=\gamma I_m$ for some $m$-th root of unity $\gamma$, as the determinant of a commutator of invertible matrices is equal to one. The exponential function $z\mapsto e^{2\pi\sqrt{-1}z}$ establishes a group isomorphism between $\bR/\bZ$ and $\bS^1\subset \bC$ with inverse $w\mapsto \frac{1}{2\pi\sqrt{-1}}\log w$. The multiplicative groups of $m$-th roots of unity and all roots of unity correspond to the subgroup $\bZ[\frac{1}{m}]/\bZ\cong \bZ/m\bZ$ and $\bQ/\bZ$ of $\bR/\bZ$ respectively under this isomorphism. We will be using this identification from now on. Hence there is a map $F_n\times F_n \to  \bZ/m\bZ\subset \bQ/\bZ$ defined by $(u,v)\mapsto \frac{1}{2\pi\sqrt{-1}}\log \gamma$. Since $f([F_n,F_n])\subset Z(U(m))$, the map factors through the abelianization of $F_n\times F_n$ and thus gives rise to a $\bZ/m\bZ$-valued skew-symmetric bilinear form $\omega_f:\bZ^n\times \bZ^n\to \bZ/m\bZ$. 

\medskip

Define a map 
$$\rho:B_n(U(m))\to  T(n,\bZ/m\bZ)$$
by $\rho(f)=(d_{ij})$, where $[f(a_i),f(a_j)]=e^{2d_{ij}\pi \sqrt{-1}}I_m$. For $D\in T(n,\bZ/m\bZ)$, let $B_n(U(m))_D=\rho^{-1}(D)$. For $f\in B_n(U(m))_D$, the ordered $n$-tuple $(f(a_1),\ldots,f(a_n))$ is said to be $D$-commuting. Note that $\rho(f)$ is the skew-symmetric matrix associated to the bilinear form $\omega_f$.  

\begin{definition} 
For any $n\times n$ matrix $D\in M_{n\times n}(\bZ/m\bZ)$, the row space $R(D)$ is the sub-module of $(\bZ/m\bZ)^n$ generated by the rows of $D$ over $\bZ$. Let $R_i(D)\subset\bZ/m\bZ$ be the
image of $R(D)$ under the projection onto the $i$-th factor.
Let $r_i(D)=|R_i(D)|$ for $i=1,\dots ,n$ and $\sigma(D)=\sqrt{|R(D)|}$. 
\end{definition}

We recall the structure of the almost commuting $n$-tuples in $U(m)$, established in \cite{ACh}, Corollary 3.6.
\begin{proposition}\label{cor:BnUgenD}
For $D\in T(n, \bZ/m\bZ)$, the space $B_n(U(m))_D$ is non-empty and path connected if $\sigma(D)$ divides $m$, and is empty otherwise. The space $B_n(U(m))$ can be expressed as a disjoint union
of path connected components
$$B_n(U(m))=\coprod\limits_{\substack{D\in T(n,\bZ/m\bZ)\\ \sigma(D)\big{|}m}}B_n((U(m))_D.$$ 
\end{proposition}

\noindent We begin our analysis by focusing on the basic case when $m=\sigma (D)$. Let $0\leq t\leq n/2$ and $d_1,d_2,\ldots,d_t\neq 0\in\bZ/m\bZ$. Define $D_n(d_1,d_2,\ldots d_t)=(d_{ij})\in 
T(n, \bZ/m\bZ)$ be the skew-symmetric matrix with
\[
d_{ij}=
\begin{cases}
d_k &\mbox{if } (i,j)=(k+t,k),1\leq k \leq t;\\
-d_k &\mbox{if } (i,j)=(k,k+t),1\leq k \leq t;\\
0 &\mbox{otherwise.}
\end{cases}
\]

\begin{lemma}\label{lemma:scalar}
	Suppose $D\in T(n, \bZ/m\bZ)$, $m=\sigma(D)$. Let $(A_1,\ldots A_n)\in B_n(U(m))_D$. If $B\in U(m)$ commutes with $A_i$ for all $i=1,\ldots,n$, then $B$ is a scalar matrix.
\end{lemma}
\begin{proof} By \cite[Proposition 4.1]{RV}, there exists $Q=(q_{ij})\in GL(n,\bZ)$ such that $Q^TDQ=D'=D_n(d_1,\ldots,d_t)$. Define $D''=D_{n+1}(d_1,\ldots,d_t)$ and $A_j'=A_1^{q_{1j}}A_2^{q_{2j}}\ldots A_n^{q_{nj}}$ for $j=1,\ldots, n$. Then $\sigma(D'')=m$ and $(A_1',\ldots A_n', B)\in B_{n+1}(U(m))_{D''}$. By \cite[theorem 3.3]{ACh}, there exists an orthonormal basis of $\bC^{m}$ consisting of eigenvectors of $B$ corresponding to a common eigenvalue. Hence, $B$ is a scalar matrix.
\end{proof}

\begin{lemma}\label{lemma:eigen}
	Suppose $D\in T(n,\bZ/m\bZ)$ and $m=\sigma(D)$. Let $(A_1,\ldots A_n)\in B_n(U(m))_D$ and $\omega=e^{2\pi\sqrt{-1}/r_j(D)}$. If $\lambda$  is an eigenvalue of $A_j$, then $\omega^q\lambda$, where  $q=0,\ldots, r_j(D)-1 $, are all the distinct eigenvalues of $A_j$.
\end{lemma}
\begin{proof}
	Let $D=(d_{ij})$. There exists $a_1,\ldots,a_n\in\bZ$ such that $a_1d_{1j}+\ldots+a_nd_{nj}=\left[-\frac{1}{r_j(D)}\right]$. Let $B=A_1^{a_1}\ldots A_n^{a_n}$. Then $BA_j=\omega^{-1} A_jB$. If $v$ is an eigenvector of $A_j$ corresponding to eigenvalue $\lambda$, then $A_jBv=\omega BA_jv=\omega B(\lambda v)=\omega \lambda Bv$ and so $\omega \lambda$ is also an eigenvalue of $A_j$. Inductively $\omega^q \lambda$ is also an eigenvalue of $A_j$ for any $q$. On the other hand, $A_j^{r_j(D)}$ commutes with $A_1,\ldots,A_n$. By lemma \ref{lemma:scalar}, $A_j^{r_j(D)}$ is a scalar matrix. Since $\lambda$ is an eigenvalue of $A_j$, $A_j^{r_j(D)}=\lambda^{r_j(D)}I_{m}$ and so any eigenvalue of $A_j$ is of the form $\omega^q \lambda$.
\end{proof}

In the situation of lemma \ref{lemma:eigen}, take an eigenvalue $\lambda_j$ of $A_j$. Define $$c_j=\frac{1}{2\pi\sqrt{-1}}\log \lambda_j\in\bR/\langle \tfrac{1}{r_j(D)}\rangle\cong\bS^1/R_j(D).$$ 
This element $c_j$ is independent of the choice of $\lambda_j$.

\begin{theorem}\label{thm:eigen}
	Let $m=\sigma(D)$. There exists a $\bT^n$-equivariant homeomorphism
	$$B_n(U(m))_D/U(m) \cong \bT^n/R(D)$$
	such that its composition with the quotient map
	$$B_n(U(m))_D\to B_n(U(m))_D/U(m) \cong \bT^n/R(D)\to \prod_{j=1}^{n} \bS^1/R_j(D)$$
	sends $(A_1,\ldots A_n)$ to $(c_1,\ldots c_n).$
\end{theorem}

\begin{proof}
If $D$ is of the special form $D=D_n(d_1,\ldots,d_t)$, the theorem follows easily from the proof of theorem 3.4 in \cite{ACh}. In the general case, let $(C_1,\ldots C_n)\in B_n(U(m))_D$. By multiplying by scalars if necessary, we assume that 1 is an eigenvalue of each $C_i$. The map $f:\bT^n\to B_n(U(m))_D$ with $f(\theta_1,\ldots,\theta_n)=(\theta_1C_1,\ldots,\theta_nC_n)$ is $\bT^n$-equivariant. Also, there exists $Q=(q_{ij})\in GL(n,\bZ)$ such that $Q^TDQ=D'=D_n(d_1,\ldots,d_t)$. Define $\phi:\bT^n\to\bT^n$ by $\phi(\theta_1,\ldots,\theta_n)=(\theta_1',\ldots,\theta_n')$ and $g:B_n(U(m))_D\to B_n(U(m))_{D'}$ by $g(A_1,\ldots,A_n)=(A'_1,\ldots,A'_n)$, where
$\theta_j' = \theta_1^{q_{1j}}\cdot \dots \cdot \theta_n^{q_{nj}}$
and $A_j'=A_1^{q_{1j}}A_2^{q_{2j}}\ldots A_n^{q_{nj}}$ for $j=1,\ldots, n$. Note that $g$ is a $\phi$-equivariant and $U(m)$-equivariant homeomorphism. Using the result for the special case $D'$, it can be deduced that the composition 
 \[ \bT^n/R(D') \overset{\bar{\phi}^{-1}}{\cong}   \bT^n/R(D)  \overset{\bar{f}}{\to}  B_n(U(m))_D/U(m) \overset{\bar{g}}{\cong} B_n(U(m))_{D'}/U(m)  \]
is a homeomorphism and so is  $\bar{f}$. It is clear that $\bar{f}^{-1}$ has the desired properties.
\end{proof}
Note that $R(D)$ is a finite subgroup of $\bT^n$, acting by translation; thus 
$\bT^n \to \bT^n/R(D)$ is a covering space and $\bT^n/R(D)$ is homeomorphic to $\bT^n$.

\medskip

For the general case $m=l\cdot\sigma (D)$, we recall that from \cite{ACh}, Corollary 3.10 that
$$B_n(U(m))_D/U(m)\cong (B_n(\sigma(D))_D/U(\sigma (D)))^l/\Sigma_l = SP^l(B_n(\sigma(D))_D/U(\sigma (D))).$$
Hence, we obtain that
\begin{theorem}\label{thm:BUnmod}
The moduli space of ordered almost commuting $n$--tuples in $U(m)$ can be expressed as a disjoint union of
symmetric products of the $n$--torus $\bT^n$:
$$B_n(U(m))/U(m) \cong \coprod\limits_{\substack{D\in T(n,\bZ/m\bZ)\\ \sigma(D)\big{|}m}}
SP^{m\over\sigma(D)}(\bT^n/R(D))$$
\end{theorem}

Now we apply the following result.
\begin{lemma}[\cite{ACGII}, Lemma 2.3]
The projection map $U(m)\to PU(m)$ induces a $U(m)$--equivariant, $\bT^n$--principal bundle
$$B_n(U(m)) \to Hom(\bZ^n, PU(m))$$ 
which gives rise to a homeomorphism
$$ \bT^n\backslash B_n(U(m))/U(m)\cong Rep (\bZ^n, PU(m)).$$
In particular it induces a bijection between $\pi_0(B_n(U(m))/U(m))$ and $\pi_0(Rep(\bZ^n, PU(m)))$.
\end{lemma}
The components $B_n(U(m))_D$ each give rise to a principal $\bT^n$--bundle, and after dividing
out by conjugation we see that the action of $\bT^n$ is given by the simultaneous
action on each unordered coordinate in the symmetric product through the homomorphism
$\bT^n\to \bT^n/R(D)$; we denote these
reduced symmetric products by
$\widehat{SP^{m\over\sigma(D)}}(\bT^n/R(D))$. Note that they are all in fact homeomorphic to the
corresponding reduced symmetric product of $\bT^n$. Combining the two results above we obtain 
\begin{theorem*}
	$$Rep(\bZ^n, PU(m)) \cong \coprod\limits_{\substack{D\in T(n,\bZ/m\bZ)\\ \sigma(D)\big{|}m}}
	\widehat{SP^{m\over\sigma(D)}}(\bT^n/R(D))$$
\end{theorem*}

The labelling of components can be understood using cohomology. As before we can identify 
$T(n,\bZ/m\bZ)$ with
$H^2(\bZ^n,\bZ/m\bZ)$ using a basis. For a projective representation $h:\bZ^n\to PU(m)$, 
there is a cohomology
class in $H^2(\bZ^n, \bZ/m\bZ)$ associated to it, defined as the pullback $\alpha= h^*(\nu)$, where
$\nu \in H^2(BPU(m), \bZ/m\bZ) \cong \bZ/m\bZ$ is the canonical generator associated to $SU(m)$. 
The component corresponding to $h$ is precisely the one labelled by a skew symmetric matrix $D$
which represents $\alpha$.
As this component is non-empty we must have
$\sigma(\alpha) \big{|} m$. 

\medskip

\begin{example}
In the case when $n=2$, $H^2(\bZ^2,\bZ/m\bZ)\cong \bZ/m\bZ$ and $\sigma (D) = |D|$, the order of $D$,
which always divides $m$ and so $N(2,m)=m$. Our decomposition can be written as
$$Rep(\bZ^2, PU(m))
 \cong \widehat{SP^m}(\bT^2)~\bigsqcup \coprod\limits_{\substack{D\in \bZ/m\bZ\\1< |D| <m}}\widehat{SP^{m\over{|D|}}}(\bT^2/R(D))~\bigsqcup~\{x_1, \dots , x_{\Phi (m)}\}$$
where $\Phi$ is Euler's function (see \cite{ACG2}, Proposition 9).
\end{example}
\begin{example}
We now consider the case when $m=p^2$ and $n\ge 4$. From the analysis in \cite{ACh}, section 3, we
see that $Rep(\bZ^n, PU(p^2))$ has 
$$	N(n,p^2)=1+\frac{(p^{n-1}-1)(p^n-1)(p^{2n+1}-p^n-p^{n-1}+p^4+p^2-1)}{(p^2-1)(p^4-1)}$$
components, of which $$u(p)=\frac{(p^{n-1}-1)(p^n-1)}{p^2-1}$$ correspond to $\sigma (D)=p$ and so we have
$$Rep(\bZ^n, PU(p^2)) \cong \widehat{SP^{p^2}}(\bT^n)~\bigsqcup
\coprod\limits_{\substack{\frac{(p^{n-1}-1)(p^n-1)}{p^2-1}}}
	\widehat{SP^{p}}(\bT^n/R(D))~\bigsqcup~ \{x_1, \dots , x_{N(n,p^2)-u(p)-1}\}$$
\end{example}



\medskip

Our analysis can be applied to describe the projective representations of any 
finitely generated abelian group. Let 
$$\Gamma=\bZ/k_1\oplus\ldots\oplus \bZ/k_s\oplus \bZ^n\cong \Tor(\Gamma)\oplus \bZ^n.$$
Define
$$B(\Gamma,U(m))= \{(A_1,\ldots, A_{s+n})\in B_{s+n}(U(m)):A_i^{k_i}=I_m \text{ for } 1\leq i\leq s\}.$$ 
It is a subspace of $B_{s+n}(U(m))$ and is invariant under the action of the subgroup
$$\{(\theta_1,\ldots,\theta_{s+n}):\theta_i^{k_i}=1 \text{ for } 1\leq i\leq s\}\cong\Tor(\Gamma)\times \bT^n\subset \bT^{s+n}.$$
This group action commutes with the conjugation action of $U(m)$. Hence, $B(\Gamma,U(m))$ is a $U(m)$--equivariant, $(\Tor(\Gamma)\times \bT^n)$--principal bundle over $Hom(\Gamma,PU(m))$. Also, there is a decomposition of
$$B(\Gamma,U(m))=\coprod B(\Gamma,U(m))_D$$ 
into a disjoint union of subspaces indexed by $D\in T(s+n,\bZ/m\bZ)$. The subspaces 
$$Hom(\Gamma, PU(m))_D\subset Hom(\Gamma, PU(m))\text{ and }Rep(\Gamma, PU(m))_D\subset Rep(\Gamma, PU(m))$$ 
can be similarly defined.

\medskip

Note that for $A_i\in U(m)$, $A_i^{k_i}=I_m$ if and only if all the eigenvalues of $A_i$ are $k_i$-th roots of unity. By Theorem \ref{thm:eigen}, we obtain the following results similar to Theorem \ref{thm:BUnmod} and Theorem A.

\begin{theorem}\label{thm:repgammaum}
	Let $\Gamma=\bZ/k_1\oplus\ldots\oplus \bZ/k_s\oplus \bZ^n\cong \Tor(\Gamma)\oplus \bZ^n.$ Then
	\[B(\Gamma, U(m))/U(m) \cong \coprod\limits_{\substack{D\in T(s+n,\bZ/m\bZ)\\ \sigma(D)\big{|}m~{\rm and}~r_i(D)\big{|}k_i}}
	SP^{m\over\sigma(D)}((\Tor(\Gamma)\times \bT^n)/R(D))\]
and	
		\[Rep(\Gamma, PU(m)) \cong \coprod\limits_{\substack{D\in T(s+n,\bZ/m\bZ)\\ \sigma(D)\big{|}m~ {\rm and} ~ r_i(D)\big{|}k_i}}
		\widehat{SP^{m\over\sigma(D)}}((\Tor(\Gamma)\times \bT^n)/R(D))\]
\end{theorem}

\begin{remark}
	Suppose $\sigma(D)\big{|}m$ and $r_i(D)\big{|}k_i$. Let $H$ be the cokernel of the composition 
	$R(D) \hookrightarrow \Tor(\Gamma)\times \bT^n \stackrel{proj}{\longrightarrow} \Tor(\Gamma)$
	of inclusion followed by projection. Then, if we focus on the images of the components associated to $D$ under
	the quotient 
	$$B(\Gamma, U(m))\to Hom (\Gamma, PU(m)),$$
	we have
	$$\pi_0(Hom(\Gamma, PU(m))_D) \cong \pi_0(Rep(\Gamma, PU(m))_D) \cong \widehat{SP^{m\over\sigma(D)}}(H).$$
The number of orbits on the right hand side can be computed using Burnside's lemma:

\[\left|\widehat{SP^{m\over\sigma(D)}}(H)\right|=\dfrac{1}{|H|}\sum_{g\in H}|SP^{m\over\sigma(D)}(H)^g|\]

where

\[|SP^{m\over\sigma(D)}(H)^g|=\begin{cases} \dbinom{\frac{m+|H|\sigma(D)}{|g|\sigma(D)}-1}{\frac{m}{|g|\sigma(D)}} & |g|\sigma(D) \mid m\\
0 & |g|\sigma(D) \nmid m \end{cases}
\]                                                 
\end{remark}

\begin{example}[Projective representations of finite abelian groups]
Our results allow us to recover results about projective representations of finite abelian groups. 
Consider a finite abelian group $\Gamma=\bZ/k_1\oplus\ldots\oplus \bZ/k_s$, where $k_i$ divides $k_{i+1}$ for $i=1,\ldots,s-1$. 
By Theorem \ref{thm:repgammaum}, the space of degree $m$ projective unitary representation modulo projective equivalence is given by
	$$Rep(\Gamma, PU(m)) \cong \coprod\limits_{\substack{D\in T(s,\bZ/m\bZ)\\ 
	\sigma(D)\big{|}m~~{\rm and}~~r_i(D)\big{|}k_i}}
\widehat{SP^{m\over\sigma(D)}}(\Gamma/R(D)).$$
If $\sigma(D)=m$,   then
$\widehat{SP^{m\over\sigma(D)}}(\Gamma/R(D))$ 
is a single point and corresponds to an irreducible projective representation. 

\medskip

Removing the restriction on the dimension, we see that 
projective equivalence classes of irreducible projective representations of $\Gamma$ are in
one--to--one correspondence with the $D\in \TnQZ$ such that $r_i(D)|k_i$ for any $i=1,\ldots,s$. Note that
since $D$ is skew-symmetric, the condition $r_i(D)|k_i$ is equivalent to $|d_{ij}|=|-d_{ji}|$ divides $k_i$ for $i<j$. 

\medskip

In terms of cohomology, this indexing can be seen to arise from pulling back using the projection $p:\bZ^s\to\Gamma$, which
yields a factorization
$$H^2(BPU(m),\bZ/m\bZ)\to H^2(\Gamma, \bZ/m\bZ)\to H^2(\bZ^s, \bZ/m\bZ).$$
Note that $p^*:H^2(\Gamma, \bQ/\bZ)\to H^2(\bZ^s, \bQ/\bZ)$ is injective, therefore if we consider all possible 
dimensions $m$ we see that the total indexing is in one-to-one correspondence with elements in $H^2(\Gamma, \bQ/\bZ)\cong H^3(\Gamma, \bZ)$.
Moreover, for the projective equivalence class corresponding to $D$, its set of representatives, up to linear equivalence, is indexed by $\Gamma/R(D)$, and each such representation has degree $\sigma(D)$.


\end{example}

\section{Projective Representations and Flat Bundles}

We now reformulate the computation of path components using homotopy theory, following the approach
in \cite{AC}, Lemma 2.5.
Recall that given a group homomorphism $h: \Gamma \to G$, it induces a continuous (pointed) 
map on classifying spaces 
$Bh: B\Gamma \to BG$.
For $G$ a compact connected Lie group, the 
correspondence 
$$Hom(\Gamma, G)\to Map_*(B\Gamma, BG)$$ 
is continuous and so induces a map on path components. As conjugation by $G$ is homotopically
trivial on $BG$, it gives rise to a map
$$\pi_0(Rep(\Gamma, G))\to [B\Gamma, BG].$$
If $X$ is a path--connected 
CW--complex with
$\pi_1(X) = \Gamma$, then composing with
the classifying map $c: X\to B\pi_1(X)$ of the universal cover 
$\tilde{X}\to X$ we obtain a map
$$\Psi^G_X: \pi_0(Rep(\pi_1(X), G))\to [X, BG]$$ 
which measures the flat principal $G$--bundles
on $X$.

\medskip

In the case when $G = PU(m)$ we have a canonical homotopy class 
$$\Omega : BPU(m)\to K(\bZ/m\bZ, 2)$$
associated to the central extension 
$$1\to \bZ /m \to SU(m) \to PU(m)\to 1.$$ 
For a CW--complex $X$, this 
gives rise to the composition
$$\pi_0(Rep(\pi_1(X), PU(m))) \xrightarrow{\Psi^{PU(m)}_X} [X, BPU(m)] \xrightarrow{\Omega} [X, K(\bZ/m\bZ,2)] = H^2(X, \bZ/m\bZ).$$
We have

\begin{theorem**}
Let $X$ denote a connected CW--complex with $\pi_1(X) = \bZ^n$; then the 
map $\Psi^{PU(m)}_X: \pi_0(Rep(\pi_1(X), PU(m)) \to [X, BPU(m)]$ is injective for all $n,m \ge 1$ and so there are 
$N(n,m)$ distinct isomorphism classes of flat principal $PU(m)$--bundles on $X$.
\end{theorem**}
\begin{proof}
First we consider the basic case $X=B\bZ^n = \bT^n$: the composition
$$\pi_0(Rep(\bZ^n, PU(m))) \xrightarrow{\Psi^{PU(m)}_{\bT^n}} [\bT^n, BPU(m)] \xrightarrow{\Omega} [\bT^n, K(\bZ/m\bZ,2)] = H^2(B\bZ^n, \bZ/m\bZ)$$
is the map on path components described in Section 2. We know that its image has precisely
$N(n,m)$ elements, corresponding to the skew symmetric matrices $D$ with $\sigma (D)$ dividing $m$.
In other words, this map distinguishes components, and is injective. If
$$c^*: H^2(B\bZ^n, \bZ/m\bZ)\to H^2(X, \bZ/m\bZ)$$ 
is induced by the classifying map, by naturality
we have
$$c^*\circ \Omega \circ \Psi^{PU(m)}_{\bT^n} = \Omega \circ \Psi^{PU(m)}_{X}.$$
From the five-term exact sequence in cohomology for the fibration
$\tilde{X}\to X\to B\pi_1X$
we infer that
$c^*: H^2(B\pi_1(X),\bZ/m\bZ)\to H^2(X, \bZ/m\bZ)$ is in fact injective. Therefore the composition
$\Omega \circ \Psi^{PU(m)}_{X}$ is injective and so is $\Psi^{PU(m)}_{X}$.
\end{proof}


\medskip

Next we study the surjectivity of the map 
\[\Psi^{PU(m)}_{\bT^n}: \pi_0(Rep(\bZ^n, PU(m)) \to [\bT^n, BPU(m)].\]
We begin with the following lemma

\begin{lemma}
Let $m\geq 2$. If $n\le 3$, $[\bT^n, BPU(m)]$ is finite, of cardinality equal to that of $\Tnm$. If $n\ge 4$, the set $[\bT^n, BPU(m)]$ has infinitely many elements. 
\end{lemma}
\begin{proof}
This follows from Woodward's classification of principal $PU(m)$--bundles for low dimensional
complexes (see \cite{W}, page 514). For $\bT^n$, $n\le 3$ he shows that the map
$$[\bT^n, BPU(m)] \to H^2(\bT^n, \bZ/m\bZ)\cong \Tnm$$ 
is a bijection. We outline a direct proof that $[\bT^n, BPU(m)]$ must be infinite 
for $n\ge 4$.
For $n=4$, there is a 
map $\Phi: \bT^4 \to \bS^4$ which induces an isomorphism on $H_4$. This arises from using the 
4-dimensional cell in a CW-complex decomposition for $\bT^4$ from its structure as a product
of circles, each having a single 0-cell and a single 1-cell. The map
$BSU(m)\to BPU(m)$ induces an isomorphism $\bZ = \pi_4(BSU(m))\cong \pi_4(BPU(m))$ and so it is possible to choose a map $\rho: \bS^4\to BPU(m)$ realizing this isomorphism.  The Hurewicz map
$\pi_4(BPU(m))\to H_4(BPU(m),\bZ)$ can be identified with the monomorphism $H_4(BSU(m),\bZ)\to H_4(BPU(m),\bZ)$, hence the composition 
$\rho\circ\Phi:\bT^4\to BPU(m)$ is a map inducing an injection on the fundamental class in $\bT^4$.
It follows that $[\bT^4, BPU(m)]$ cannot be finite. For 
$n\ge 4$, we can use the split surjection $[\bT^n, BPU(m)]\to [\bT^4, BPU(m)]$ to verify the claim.
\end{proof}

\begin{proposition}\label{prop:surj}
Let $m\geq 2$, then $\Psi^{PU(m)}_{\bT^n} : \pi_0(Rep(\bZ^n, PU(m)) \to [\bT^n, BPU(m)]$ is surjective if and
only if $n\leq 3$.
\end{proposition}  
\begin{proof}
From the definition in cohomology (and using an appropriate basis when $n=3$) it is easy to see that for $n=1,2,3$, we have that 
$\{D\in\Tnm:\sigma(D)~\big{|}~m\}=\Tnm$.
Thus we conclude that $\Omega\circ\Psi^{PU(m)}_{\bT^n}$ is a bijection and
therefore by a cardinality argument so is $\Psi^{PU(m)}_{\bT^n}$ for $n=1,2,3$. 
For $n\ge 4$ we have verified that $[\bT^n, BPU(m)]$ is infinite, whence the result follows.
\end{proof}


\begin{corollary}
There exists a principal $PU(m)$-bundle on the $n$--torus $\bT^n$ which does not admit 
a flat structure if and only if $n\ge 4$.
\end{corollary}

		


\section*{Acknowledgements}
The first author was funded by NSERC. We are grateful to 
M. Bergeron, J.M. G\'omez, Z. Reichstein and B. Williams for their helpful comments.

\bibliographystyle{plain} 

\begin{thebibliography}{1}

\bibitem{ACh}
Alejandro Adem and Man Chuen Cheng.
\newblock Representation spaces for central extensions and almost commuting unitary matrices. 
\newblock{\em J. London Math. Soc.} 94 (2): pp 503-524 (2016). 

\bibitem{ACGII}
Alejandro Adem, Frederick~R. Cohen, and Jos{{\'e}}~Manuel G{{\'o}}mez.
\newblock Stable splittings, spaces of representations and almost commuting elements in Lie groups
\newblock {\em Math. Proc. Camb. Phil. Soc.}, 149, 455-490, 2010.

\bibitem{ACG2}
Alejandro Adem, Frederick~R. Cohen, and Jos{{\'e}}~Manuel G{{\'o}}mez.
\newblock Commuting elements in central products of special unitary groups
\newblock Proc. Edinburgh Math. Soc. (Series 2) 56(1):1--12, 2013. 

\bibitem{AC}
Alejandro Adem and Frederick~R. Cohen.
\newblock Commuting elements and spaces of homomorphisms.
\newblock {\em Math. Ann.}, 338(3):587--626, 2007.

\bibitem{Bergeron}
Maxime Bergeron.
\newblock The topology of nilpotent representations in
reductive groups and their maximal compact subgroups.
\newblock {\em Geom. Topol.}, 19(3):1383--1407, 2015.

\bibitem{Borel} 
\newblock Armand Borel.
\newblock Sous-groupes commutatifs et torsion des groupes de Lie compacts connexes.
\newblock {\em Tohoku Math. J.}, (2) Vol. 13, Number 2 (1961), 216-240. 


\bibitem{BFM}
Armand Borel, Robert Friedman, and John~W. Morgan.
\newblock Almost commuting elements in compact {L}ie groups.
\newblock {\em Mem. Amer. Math. Soc.}, 157(747):x+136, 2002.

\bibitem{RV}
Zinovy Reichstein and Nikolaus Vonessen.
\newblock Rational central simple algebras.
\newblock {\em Israel J. Math.}, 95:253--280, 1996.

\bibitem{W} 
L. M. Woodward
\newblock The Classification of principal $PU(n)$--bundles over a 4--complex
\newblock{J. London Math. Soc.}, (2), 25 (1982), 513-524.

\end{thebibliography}

\end{document}